\documentclass{article}
\usepackage{graphicx} 
\usepackage{tikz}
\usepackage{array}
\usepackage[margin=1in]{geometry}
\usepackage{amsmath} 
\usetikzlibrary{positioning, calc}

\title{Curvature of basis exchange walks}
\author{Isabel Detherage\footnote{Supported by NSF CCF-2420130} \\ UC Berkeley}
\date{\today}
\usepackage{graphicx} 
\usepackage{amsmath,amssymb,amsthm}
\usepackage{algpseudocode}
\usepackage{algorithm}
\usepackage{fullpage}
\usepackage{hyperref}
\usepackage{xcolor}
\usepackage{mleftright}
\usepackage{enumitem}
\usepackage{array}

\newcommand{\squaretikz}[6]{
\begin{tikzpicture}[scale=0.4]
    \draw[#1, thick] (0,0) -- (0,1); 
    \draw[#2, thick] (0,1) -- (1,1); 
    \draw[#3, thick] (0,0) -- (1,0); 
    \draw[#4, thick] (1,0) -- (1,1); 

    \draw[#5, thick] (0,0) -- (1,1); 
    \draw[#6, thick] (1,0) -- (0,1); 
\end{tikzpicture}
}

\newtheorem{theorem}{Theorem}[section]
\newtheorem{lemma}[theorem]{Lemma}
\newtheorem{proposition}[theorem]{Proposition}
\newtheorem{corollary}[theorem]{Corollary}

\theoremstyle{definition}
\newtheorem{definition}{Definition}[section]

\newtheorem{remark}{Remark}

\algdef{SE}[REPEATN]{RepeatN}{End}[1]{\algorithmicrepeat\ #1 \textbf{times}}{\algorithmicend}

\begin{document}

\maketitle

\begin{abstract}
    We prove both lower and upper bounds on the Ollivier-Ricci curvature of the basis exchange walk on a matroid. We give several examples of non-negatively curved basis exchange walks and negatively curved basis exchange walks.
\end{abstract}

\section{Introduction}
In this note, we consider the discrete Ollivier-Ricci curvature of a specific class of Markov chains, basis exchange walks on matroids. Ollivier--Ricci curvature on metric spaces was first introduced in \cite{ORcurv}, as an analog to Ricci curvature on manifolds. This definition allowed for several natural extensions of results in geometry, such as Bonnet-Myers theorem and a lower bound on the spectral gap (both of which apply only to spaces with non-negative curvature). Notably, non-negative discrete curvature also implies bounds on the mixing time of a Markov chain (see \cite{LPW} on path coupling for more details).

Basis exchange walks on matroids are a well-studied class of Markov chains which exhibit many desirable properties: lower bounded spectral gap and lower bounded modified log Sobolev constant \cite{kldivbew}, spectral independence, and fast mixing \cite{ALOVIV}. We note that all of these properties are \textit{implied} by non-negative Ollivier-Ricci curvature and pose the natural follow-up question: 

\smallskip
\centerline{\text{Do all basis exchange walks on matroids have non-negative Ollivier-Ricci curvature?}}

\smallskip
We answer this question negatively: despite the many desirable properties uniformly exhibited by basis exchange walks, there exist both negatively and non-negatively curved basis exchange walks. Furthermore, we show even among standard classes of matroids, such as linear matroids or graphic matroids, there exist both matroids with negatively and non-negatively curved basis exchange walks. This work reinforces the idea that non-negative curvature is a somewhat delicate phenomenon, which cannot be guaranteed by a strongly log-concave stationary distribution (a feature of all basis exchange walks) or the representability of the matroid.

We highlight in particular the implication this has for the connection between spectral independence and Ollivier-Ricci curvature, both local certificates of expansion. Concurrent works \cite{couplingtospecind} and \cite{couplingliu} prove non-negative Ollivier-Ricci curvature \textit{implies} spectral independence (see either work for precise, quantified statements). The validity of the converse to this statement was not previously known, but this work disproves it: basis exchange walks are uniformly 1-spectrally independent, but we establish several examples with negative Ollivier-Ricci curvature.

\subsection{Techniques and overview}
To prove the lower bound on curvature (Theorem \ref{thm:lb}), we analyze a natural coupling of two walks starting from adjacent bases: in the down step, both walks drop the same element if possible, and in the up-step, they add the same element if possible. For the upper bound (Theorem \ref{thm:ub}), we prove a lower bound on the expected distance of two walks starting from adjacent bases for \textit{all} possible couplings. We do so by exhaustively considering \textit{all} events where the distance can drop to 0, and upper bounding the probability that the distance goes to 0 under any coupling. We then consider events where the distance \textit{must} be at least 2 for any coupling and lower bound this probability. 

In Section \ref{sec:pre} we give relevant preliminaries. In Section \ref{sec:lb} we give a coupling of the down-up walk and use this to prove a lower bound on curvature; we use the same coupling in Section \ref{sec:nonnegex} to show non-negative curvature for several examples. In Section \ref{sec:ub} we prove an upper bound on curvature, and in Section \ref{sec:negex} we give examples of basis exchange walks with negative curvature.  

\section{Preliminaries}\label{sec:pre}
\subsection{Matroids and basis exchange walk}

\begin{definition}[Matroid]
    A \textit{matroid} $\mathcal{M} = (E, \mathcal{B})$ is a set $E$ and a non-empty collection of its subsets $\mathcal{B}$, called \textit{bases}, satisfying the following: \begin{itemize}
        \item no proper subset of a basis is itself a basis, and
        \item for any $B_1, B_2 \in \mathcal{B}$, if $b_1 \in B_1 \backslash B_2$, there exists a $b_2 \in B_2 \backslash B_1$ such that $B_1 - b_1 + b_2 \in \mathcal{B}$.
    \end{itemize}
    The second condition is known as the \textit{basis exchange property}.
\end{definition}

\begin{definition}[Rank]
    The \textit{rank} of a matroid $\mathcal{M} = (E, \mathcal{B})$ is the size of any basis $B \in \mathcal{B}$. Note that by the basis exchange property, all bases must be the same size.
\end{definition}

A particularly interesting class of matroids are those induced by a graph $G$, known as graphic matroids. We consider several graphic matroids as examples, so we define them here.
\begin{definition}[Graphic matroids]
    A \textit{graphic matroid} induced by a graph $G$ is a matroid whose basis sets are the spanning forests of $G$. If $G$ is connected, the basis sets are the spanning trees of $G$.
\end{definition}

\begin{definition}[Down-up walk]
    The \textit{down-up walk} or the \textit{basis exchange walk} is a Markov chain on the bases of a matroid. Starting from a basis $S \in \mathcal{B}$, the random walk proceeds by dropping an element $u \in S$ uniformly at random, then moving to a basis containing $S-u$ uniformly at random.
\end{definition}
We note that the above describes the basis exchange walk on a set of bases given by the uniform distribution (i.e., the uniform distribution over all bases will be the stationary distribution). Given a distribution $\pi$ over $\mathcal{B}$, one can define the corresponding down-up walk, but for the purposes of this note, we will only be concerned with the uniform case. 

\subsection{Ollivier--Ricci curvature}

We are interested in the curvature of this walk, in particular the \textit{Ollivier--Ricci curvature}. In order to define Ollivier--Ricci curvature, we first define the Wasserstein distance.

\begin{definition}[Wasserstein distance]
    Let $\mu, \nu$ be two probability distributions over some metric space $(X, d)$. The \textit{Wasserstein distance} between $\mu$ and $\nu$ is defined as \[\mathcal{W}(\mu, \nu) = \min_{X \sim \mu, Y \sim \nu} \mathbb{E} d(X, Y),\] where the minimum is over all possible couplings of $\mu$ and $\nu$.
\end{definition}

\begin{definition}[Ollivier--Ricci curvature]
    For a Markov chain $P$ over a metric space $(X, d)$, we define the \textit{Ollivier--Ricci curvature} of $P$ as the minimum $\kappa$ such that for all $S, T \in X$ \[\mathcal{W}(P(S, \cdot), P(T, \cdot)) \leq (1-\kappa)\cdot d(S,T).\] 
\end{definition} Note if $d$ is the metric induced by $P$, i.e., $d(S,T) := \inf \{t : P^t(S,T) > 0\}$, we can consider the minimum over only \textit{adjacent} $S,T \in X$ in the above definition (see \cite{ORcurv} for further details). Throughout, we will take $d$ to be the metric induced by $P$.

\section{A candidate coupling and a lower bound}\label{sec:lb}
\subsection{Coupling}
We define a candidate coupling, which we call the \textit{down-step coupling}, for the down-up walk starting from any adjacent bases $S, T$. The following will be useful to denote all possible ``up-steps" of a given down-step:
\begin{definition}
    If $B$ is a basis and $u \in B$, let $N(B - u) = \{x \in E : B - u + x \text{ is a basis}\}.$
\end{definition}
\begin{definition}[Down-step coupling]
    Let $S = \{s, u_1, \hdots, u_{k-1}\}$, $T = \{t, u_1, \hdots, u_{k-1}\}$ be bases of a rank-$k$ matroid. Let $P$ be the basis exchange walk, $X \sim P(S, \cdot)$, and $Y \sim P(T, \cdot)$. We couple $X$ and $Y$ by first matching the down-steps when possible, then matching the up-steps when possible. More concretely: \begin{itemize}
        \item If $S$ drops $s$, then $T$ drops $t$. Since $S-s = T-t$, we can always couple the up steps to be the same. 
        \item If $S$ drops $u_i$, then $T$ drops $u_i$. Without loss of generality, suppose \[\#N(S-u_i) \geq \#N(T-u_i).\] We couple the up steps in the following way: \begin{itemize}
            \item If $S-u_i$ adds $t$, $T-u_i$ adds $s$ (note if $t \in N(S-u_i)$, then $s \in N(T-u_i)$ since $S-u_i +t = T-u_i + s$).
            \item If $S-u_i$ adds $v \in N(S-u_i) \cap N(T-u_i)$, $T-u_i$ adds $v$.
            \item If $S-u_i$ adds $v \notin N(S-u_i) \cap N(T-u_i)$ with $v \neq t$, then $T-u_i$ adds a random element from $N(T-u_i)$ with any distribution preserving the desired marginal across $N(T-u_i)$.
        \end{itemize}
    \end{itemize} The distribution mentioned in the last line is simply a distribution such that the up-step for $T-u_i$ is uniform across $N(T-u_i)$; such a distribution will always be possible given that $\#N(S-u_i) \geq \#N(T-u_i)$. Note that with this coupling, the walks are at distance at most $2$ by the basis exchange property.
\end{definition}

\begin{remark}
    We remark that while this coupling might not be optimal, coupling different down steps for $S$ and $T$ can only decrease $\mathbb{E}d(X,Y)$ in a small number of cases. For example, consider $S = \{s, u, u'\}$ and $T = \{t,u,u'\}$. Suppose $S$ drops $u$ and adds $a$, while $T$ drops $u'$ and adds $b$, so $S$ moves to $\{s,u,a\}$ and $T$ moves to $\{t,b,u'\}$. We are only in a better position (i.e., at a distance less than 2) if: (1) $a = t$ and $b = s$, in which case the distance is 1, or (2) $a = u'$ and $b = u$, in which case the distance is also 1. We see that for most possible up-steps, it is more advantageous to couple based on the down-step.
    
    We make this more precise with an upper bound on curvature in Section \ref{sec:ub}.
\end{remark}
\paragraph{Example:} We give an example of this coupling on the graphic matroid induced by $K_4$. Let $S$ and $T$ be the following spanning trees, respectively:

\begin{center}
    
\begin{tikzpicture}[scale = 0.5]

\coordinate (A) at (0,0);
\coordinate (B) at (1,0);
\coordinate (C) at (1,1);
\coordinate (D) at (0,1);

\draw[very thick] (A) -- (D); 
\draw[very thick] (D) -- (C); 
\draw[very thick] (A) -- (B); 

\draw[dotted, thick] (B) -- (C); 

\draw[dotted, thick] (A) -- (C);
\draw[dotted, thick] (B) -- (D);

\end{tikzpicture} \qquad
\begin{tikzpicture}[scale = 0.5]
\coordinate (A) at (0,0);
\coordinate (B) at (1,0);
\coordinate (C) at (1,1);
\coordinate (D) at (0,1);

\draw[very thick] (A) -- (D); 
\draw[very thick] (D) -- (C); 
\draw[dotted, thick] (A) -- (B); 

\draw[dotted, thick] (B) -- (C); 

\draw[dotted, thick] (A) -- (C);
\draw[very thick] (B) -- (D);

\end{tikzpicture}
\end{center} The following table shows how walks would move from $S$ and $T$ in the down-step coupling, where $s$ is the bottom edge and $t$ is the diagonal edge from the upper left to bottom right. For clarity, we group outcomes under which element is dropped in the down step of the down-up walk.

\begin{figure}[h]

\renewcommand{\arraystretch}{2.2}

\begin{center}
\begin{tabular}{|>{\centering\arraybackslash}m{1cm}
                |>{\centering\arraybackslash}m{0.7cm}
                |>{\centering\arraybackslash}m{0.7cm}
                |>{\centering\arraybackslash}m{0.7cm}
                |>{\centering\arraybackslash}m{0.7cm}
                |>{\centering\arraybackslash}m{0.7cm}
                |>{\centering\arraybackslash}m{0.7cm}
                |>{\centering\arraybackslash}m{0.8cm}
                |>{\centering\arraybackslash}m{1cm}
                |>{\centering\arraybackslash}m{0.7cm}
                |>{\centering\arraybackslash}m{1.7cm}|}
\hline
initial state & \multicolumn{3}{c|}{drop $s$/drop $t$} & \multicolumn{3}{c|}{drop top edge} & \multicolumn{4}{c|}{drop left edge} \\
\hline
\squaretikz{}{}{}{dotted}{dotted}{dotted}
& \squaretikz{}{}{}{dotted}{dotted}{dotted}
& \squaretikz{}{}{dotted}{}{dotted}{dotted} 
& \squaretikz{}{}{dotted}{dotted}{dotted}{} 
& \squaretikz{}{}{}{dotted}{dotted}{dotted} 
& \squaretikz{}{dotted}{}{}{dotted}{dotted} 
& \squaretikz{}{dotted}{}{dotted}{}{dotted} 
& \squaretikz{dotted}{}{}{dotted}{dotted}{}
& \squaretikz{}{}{}{dotted}{dotted}{dotted} 
& \squaretikz{dotted}{}{}{dotted}{}{dotted} 
& \squaretikz{dotted}{}{}{}{dotted}{dotted} \\
\hline
\squaretikz{}{}{dotted}{dotted}{dotted}{} 
& \squaretikz{}{}{}{dotted}{dotted}{dotted}
& \squaretikz{}{}{dotted}{}{dotted}{dotted} 
& \squaretikz{}{}{dotted}{dotted}{dotted}{} 
& \squaretikz{}{}{dotted}{dotted}{dotted}{} 
& \squaretikz{}{dotted}{dotted}{}{dotted}{} 
& \squaretikz{}{dotted}{dotted}{dotted}{}{} 
& \squaretikz{dotted}{}{}{dotted}{dotted}{}
& \squaretikz{}{}{dotted}{dotted}{dotted}{}
& \squaretikz{dotted}{}{dotted}{dotted}{}{} 
& \squaretikz{dotted}{}{}{dotted}{dotted}{} \hspace{1mm} \squaretikz{dotted}{}{dotted}{dotted}{}{} \hspace{1mm} \squaretikz{}{}{dotted}{dotted}{dotted}{} \\
\hline
$\mathbb{P}$ & $\frac{1}{9}$ & $\frac{1}{9}$ & $\frac{1}{9}$ & $\frac{1}{9}$ & $\frac{1}{9}$ & $\frac{1}{9}$ & $\frac{1}{12}$ & $\frac{1}{12}$ & $\frac{1}{12}$ & $\frac{1}{36}$ \footnotesize{each pair} \\
\hline
& \multicolumn{3}{c|}{\small{add same edge}} & \multicolumn{3}{c|}{\small{add same edge}} & \footnotesize{add $t$ / add $s$} & \footnotesize{add left edge} & \footnotesize{add diag.} &  \\
\hline
\end{tabular}
\end{center}
\caption{The down-step coupling for two neighboring states in the graphic matroid induced by $K_4$.}
\end{figure}

\subsection{Lower bound}
Our first theorem provides a lower bound on the curvature of the basis exchange walk in terms of the rank and the size of the ground set.

\begin{theorem}\label{thm:lb}
    The basis exchange walk of a rank-$k$ matroid over a set of size $n > k+1$ satisfies \[\kappa \geq -1 + \frac{2}{k} + \frac{3(k-1)}{k(n- k+ 1)}.\] If $n=k+1$, the basis exchange walk satisfies $\kappa \geq 1/k$.
\end{theorem}
\begin{proof}
    We show that the above coupling gives us the desired bound. 
    
    First note if $S$ drops $s$ and $T$ drops $t$ (which occurs with probability $\frac{1}{k}$) we have \[\mathbb{E} \left (d(X,Y) \hspace{1mm} | \hspace{1mm} s,t \text{ dropped in down step}\right ) = 0.\]

    Now we consider what happens when we drop $u_i$. Letting $U' = \{u_1,\hdots, u_{i-1}, u_{i+1}, \hdots, u_{k-1}\}$, we have two cases:
    \paragraph{Case 1:} $U' + s + t \notin \mathcal{B}$. 
    
    We claim $N(S - u) = N(T-u)$ by the basis exchange property. Let $x \in N(S-u)$ and $y \in N(T-u)$, so \[S - u + x = U'+s+x \in \mathcal{B}, \qquad  T - u + y = U' + t + y \in \mathcal{B}.\] Applying the basis exchange property to $U'+s+x, U' + t + y,$ and $x$, we must have $U' + s + y \in \mathcal{B}$ since $U' + s + t \notin \mathcal{B}$, so $y \in N(S-u)$. Similarly, we can see that $U' + t + x \in \mathcal{B}$, so $x \in N(T-u)$. 
    
    Since $N(S-u) = N(T-u)$, we can couple these walks to have the same up-step and in this case \[\mathbb{E}\left (d(X,Y) \hspace{1mm} | \hspace{1mm} u_i \text{ dropped} \right )= 1.\]
    
    \paragraph{Case 2:} $U' + s + t \in \mathcal{B}$

    In this case, we can couple the walks to both move to $U' + s + t$ with some probability and the distance will be 0. Note that $N(S-u), N(T-u) \leq n-(k-1)$ since $(S -u) \cap N(S-u) = \emptyset$ (and similarly for $T$), so \[\min\{\mathbb{P}(U'+s \rightarrow U' + s + t), \mathbb{P}(U' + t \rightarrow U' + s + t)\} \geq \frac{1}{n-k+1}\] where $\mathbb{P}(U'+s \rightarrow U' + s + t)$ denotes the probability of adding $t$ on the up-step, given $u_i$ was dropped on the down-step (i.e., $1 / N(S-u_i)$). Similarly, we can couple the walks to stay at $S$ and $T$ respectively, which also occurs with probability $\geq \frac{1}{n-k+1}$, and the distance is 1. Otherwise, the distance is always less than or equal to 2 since \[U' + s + x \sim U' + s + t \sim U' + t + y.\] All together then we have \[\mathbb{E}\left (d(X,Y) \hspace{1mm} |\hspace{1mm} u_i \text{ dropped}\right ) \leq 2 \cdot \left (1 - \frac{2}{n-k+1} \right ) + \frac{1}{n-k+1} = 2 - \frac{3}{n-k+1}.\] 

    \medskip
    If $n > k+1$, in either case we have \[\mathbb{E}\left (d(X,Y) \hspace{1mm}|\hspace{1mm} u_i \text{ dropped} \right ) \leq 2 - \frac{3}{n-k+1}.\] Then we have \[\mathcal{W}(P(S, \cdot), P(T, \cdot)) \leq \frac{(k-1)}{k} \cdot \left (2 - \frac{3}{n-k+1} \right ).\] By definition of $\kappa,$ we have: \begin{align}
        \kappa &\geq 1 - \frac{(k-1)}{k}\cdot \left (2 - \frac{3}{n-k+1} \right ) \\
        & = 1 - \frac{2(k-1)}{k} + \frac{3(k-1)}{k(n-k+1)} \\
        & = -1 + \frac{2}{k} + \frac{3(k-1)}{k(n-k+1)}.
    \end{align}

    If $n = k+1$, then in either case $\mathbb{E}(d(X,Y) \hspace{1mm} | \hspace{1mm} u_i \text{ dropped}) \leq 1$, so we have $\kappa \geq 1/k.$
\end{proof}

By more carefully tracking which case we are in, we can sharpen this bound. We define the set of elements we can drop that would place us in case two from above:

\begin{definition}
    Let $J = \{u : t \in N(S-u) \text{ and } u\neq s\}.$
\end{definition}
\begin{corollary} The basis exchange walk over any matroid satisfies
    \[\kappa \geq \min_{S\sim T}\left \{ -\frac{\#J}{k} + \frac{1}{k} + \frac{1}{k}\sum_{u \in J} \frac{2 + \#(N(S-u) \cap N(T-u))}{\max\{\#N(S-u), \#N(T-u)}\}\right \}.\]
\end{corollary} 
\begin{proof}
    This follows immediately from the down-step coupling and definitions. 
\end{proof}

While this expression is not particularly friendly, it makes explicit the idea that if there is a high overlap between $N(S-u)$ and $N(T-u)$ for all $u$, the basis exchange walk will be positively curved. In the following examples, we directly analyze the coupling, which is equivalent to applying this corollary (although we refrain from referring to this expression).

\section{Examples of non-negatively curved basis exchange walks}\label{sec:nonnegex}
\subsection{Non-negative curvature from smallness}
We have the following corollaries to Theorem \ref{thm:lb} for rank-2 matroids, matroids over small ground sets, and graphic matroids over small graphs:
\begin{corollary}\label{cor:rank2}
    The basis exchange walk on any rank-$2$ matroid is positively curved.
\end{corollary}
\begin{corollary}\label{cor:size7}
    The basis exchange walk on any matroid with $n \leq 7$ is non-negatively curved.
\end{corollary} 
\begin{corollary}
    The basis exchange walk on any graphic matroid induced by a simple graph $G = (V,E)$ with $\#V \leq 4$ is non-negatively curved. 
\end{corollary}

\subsection{Uniform matroid}
\begin{lemma}
    The basis exchange walk over the rank-$k$ uniform matroid with size $n$ ground set satisfies \[\kappa \geq 1 - \frac{(k-1)(n-k)}{k(n-k+1)}.\]
\end{lemma}
\begin{proof}
    The down-step coupling gives \[\mathbb{E}d(X,Y) = \frac{1}{k}\cdot 0 + \frac{k-1}{k}\left (\frac{1}{n-k+1}\cdot 0 + \frac{n-k}{n-k+1}\cdot 1 \right ).\]
\end{proof}

\subsection{V\'{a}mos matroid}
The \textit{V\'{a}mos matroid} is a rank-$4$ matroid over a ground set of size $8$ that cannot be represented as a matrix over any field. However, this matroid is close, in some sense, to a uniform matroid; it includes 65 of the possible 70 subsets of size $4$ as bases. The excluded sets are the shaded parallelograms in Figure 2\footnote{Provided by David Eppstein, CC0.}.
\begin{figure}[h!]
    \centering
    \includegraphics[scale = 0.1]{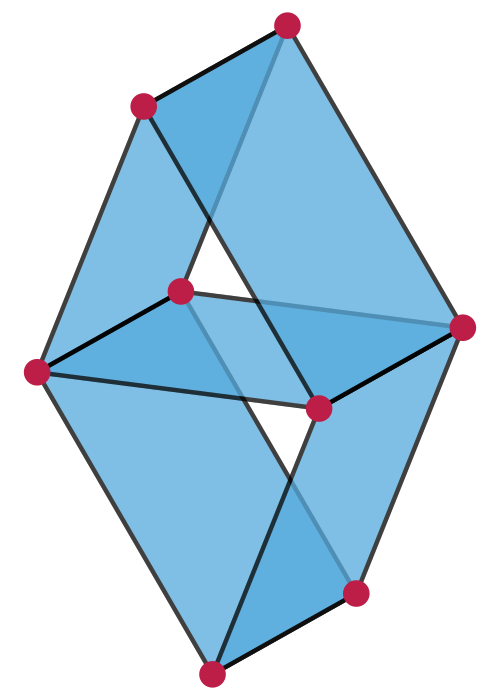}
    \caption{Dependent sets in the V\'{a}mos matroid.}
    \label{fig:placeholder}
\end{figure}

The following lemma shows that the V\'{a}mos matroid is positively curved. We can interpret this as a consequence of how close the V\'{a}mos matroid is to the rank-$4$ uniform matroid over $8$ elements.

\begin{lemma}
    The basis exchange walk over the V\'{a}mos matroid is positively curved.
\end{lemma}
\begin{proof}
    Again we show the down-step coupling demonstrates positive curvature.
    
    Let $S = \{a,b,c,s\}$ and $T = \{a,b,c,t\}$ for some distinct elements $a,b,c,s,t.$ For $u \neq s, t$ we consider the possibilities for $N(S-u)$ and $N(T-u)$. By definition of the V\'{a}mos matroid, for any $S \in \mathcal{B}$ and any $u \in S$, $\#N(S-u) = 5$ or $4$ (i.e., it either contains all the elements not in $S -u$ or one less). We compute the expected distance in each case.

    Suppose $\#N(S-u) = \# N(T-u) = 5$. Since $t \in N(S-u)$, the down-step coupling in this case gives \[\mathbb{E}\left (d(X,Y) \hspace{1mm} | \hspace{1mm} u \text{ dropped} \right ) = \frac{4}{5}.\]

    Suppose $\#N(S-u) = 5$ and $\#N(T-u) = 4$. Again since $t \in N(S-u)$, in this case we have \[\mathbb{E}\left (d(X,Y) \hspace{1mm} | \hspace{1mm} u \text{ dropped} \right ) \leq \frac{1}{5} \cdot 0 + \frac{3}{5}\cdot 1 + \frac{1}{5} \cdot 2 = 1.\]

    Suppose $\#N(S-u) = \#N(T-u) = 4$. As seen in the proof of Theorem \ref{thm:lb}, if $S-u + t \notin \mathcal{B}$, we can couple so the distance is always 1, so we assume $t \in S-u$ and $s \in T-u$. Then we must have $\{u, d\} \subseteq N(S-u) \cap N(T-u)$ for some $d$. Conditioned on being in this case, we have \[\mathbb{E}\left (d(X,Y) \hspace{1mm} | \hspace{1mm} u \text{ dropped} \right ) \leq \frac{1}{4} \cdot 0 + \frac{1}{2}\cdot 1 + 2 \cdot \frac{1}{4} = 1.\]

    In all cases with $u \neq s, t$ the expected distance is less than or equal to 1, so we see the basis exchange walk on this matroid has positive curvature.
\end{proof}

\subsection{Graphic matroids}
Now we consider a specific class of graphic matroids that induce positively curved basis exchange walks. This shows that in addition to having non-negatively curved basis exchange walks on graphic matroids induced by sufficiently small graphs, one can have non-negatively curved basis exchange walks on graphic matroids induced by arbitrarily large graphs.
\begin{lemma}
    Basis exchange walks over graphic matroids induced by graphs with edge-disjoint cycles are positively curved.
\end{lemma}

\begin{proof}
    We express the graph as a union of cycles, $C_i$, and edges with effective resistance 1, $G'$: \[G = C_1 + \hdots + C_k + G'.\] Any spanning tree $S$ of the graph must be a union of spanning trees $S_i$ of each cycle and $G'$ so we have \[S = S_1 + \hdots + S_k + G'.\] Any adjacent spanning tree $T$ must differ in only one spanning tree of one cycle. Without loss of generality we assume the different spanning tree is in the cycle labeled as $C_1$, so \[T = T_1 + S_2 + \hdots + S_k + G'.\] 

    We must show for any $u \in S \cap T$, we can couple walks from $S$ and $T$ conditioned on dropping $u$ so that the expected distance is not greater than 1. We consider different cases for $u$. 
    
    Suppose $u \in G'$. For any $B \in \mathcal{B}$ we have $N(B-u) = \{u\}$ (i.e, the only edge that can be added on the up step is $u$). For this case, we have \[\mathbb{E}\left (d(X,Y) \hspace{1mm} | \hspace{1mm} u \text{ dropped} \right ) = 1.\] 
    
    Suppose $u \in S_i$ with $i \neq 1$, then $N(S-u) = N(T-u) = \{u,u'\}$ where $\{u'\} = C_i \backslash S_i$. We can always add the same element to each of $S-u$ and $T-u$, so we again have \[\mathbb{E}\left (d(X,Y) \hspace{1mm} | \hspace{1mm} u \text{ dropped} \right ) = 1.\] 
    
    Suppose $u \in S_1$, then $N(S-u) = \{u, t\}$ and $N(T-u) = \{u, s\}$ where $t = C_1 \backslash S_1$ and $s = C_1 \backslash T_1$ (similar to previous notation, $s$ is simply the edge included in $S$ but not in $T$). Since $S_1 = C_1 - t$ and $T_1 = C_1 - s$, $S_1 - u + t = T_1 - u + s$. We couple by adding $u$ to both $S$ and $T$ or adding $t$ to $S$ and $s$ to $T$. In this case, we have \[\mathbb{E}\left (d(X,Y) \hspace{1mm} | \hspace{1mm} u \text{ dropped} \right ) = 1/2.\] 

    This covers all cases, so the basis exchange walk is positively curved.
\end{proof}

\section{Upper bound on curvature}\label{sec:ub}

We now give an upper bound on the curvature of the basis exchange walk (we note the main function of this bound will be to prove specific examples have negative curvature). In order to state this bound, we fix any $S \sim T$ and define the following notation.
\begin{definition}
    Let $J = \{u : t \in N(S-u) \text{ and } u\neq s\}.$
\end{definition}

\begin{definition}
    For any $u \in J$, let $A_u = (N(S-u) - t) \backslash N(T-u)$, i.e., the elements that are not $t$ and that can be added to $S-u$ but \textit{cannot} be added to $T-u$.
\end{definition}

\begin{remark}
    These sets will be to identify neighbors of $S$ that are all far from all but a small number of neighbors of $T$ (see Proposition \ref{prop:neighfar} for a more precise statement). If $t \notin N(S-u)$, then as seen in the proof of Theorem \ref{thm:lb}, $N(S-u) = N(T-u)$. In this case, every neighbor of $S$ induced by dropping $u$ has a unique neighbor of $T$ induced by dropping $u$.
\end{remark}

\begin{theorem}\label{thm:ub}
    Consider any bases $S \sim T$ of a matroid $\mathcal{M}$. Then the basis exchange walk on $\mathcal{M}$ satisfies \[\kappa \leq \frac{1}{k} + \frac{1}{k}\cdot\sum_{u\in J}\left (\frac{1}{\#N(T-u)} - \frac{\#A_u}{\#N(S-u)} \right ). \]
\end{theorem}

\begin{proof}

Let $S$ and $T$ be adjacent bases with $S = \{s, u_1, \hdots, u_{k-1}\}$ and $T = \{t, u_1, \hdots, u_{k-1}\}$. Let $X \sim P(S,\cdot)$ and $Y \sim P(T, \cdot)$. Note for any coupling we have \begin{align}
    \mathbb{E}d(X,Y) &\geq 0 \cdot \mathbb{P}(d(X,Y) = 0) + 1 \cdot \mathbb{P}(d(X,Y) = 1) + 2 \cdot \mathbb{P}(d(X,Y) \geq 2) \\
    &= (1-\mathbb{P}(d(X,Y) = 0) - \mathbb{P}(d(X,Y) \geq 2)) + 2 \cdot \mathbb{P}(d(X,Y) \geq 2) \\
    &= 1 - \mathbb{P}(d(X,Y) = 0) + \mathbb{P}(d(X,Y) \geq 2).
\end{align} 
Recalling the definition of curvature, we see then \[\kappa \leq \mathbb{P}(d(X,Y) = 0) - \mathbb{P}(d(X,Y) \geq 2).\] We will upper bound this quantity for any coupling.

We start by considering events where $d(X,Y) \geq 2$. In particular, we want to identify neighbors of $S$ that are from most neighbors of $T$. The following proposition shows that any neighbor of $S$ of the form $S-u + a$ for some $u \in J$ and some $a \in A_u$ is far from almost all neighbors of $T$.

\begin{proposition}\label{prop:neighfar}
    For any $u \in J$ and $a \in A_u$, $S-u+a$ is at distance at least 2 from any neighbor of $T$ aside from $T-u+s,$ $T-t+s,$ or $T-t+a$ (if these sets are bases).
\end{proposition}
\begin{proof}
    Without loss of generality, let $u = u_1$. Then $S-u+a = \{s, a, u_2, \hdots, u_{k-1}\}$. We proceed by cases on which element $T$ drops.

    \smallskip \noindent \textbf{Case 1:} $T$ drops $u_1$ and adds some element $b$, so $T-u_1 + b = \{t, b, u_2, \hdots, u_{k-1}\}$. 
    
    \noindent Recall by definition of $A_{u_1}$, $b \notin N(T-u_1)$ so $b \neq a$. Since $a \neq u_1,t$, these are at distance 2 unless $b = s$.

    \smallskip \noindent \textbf{Case 2:} $T$ drops $t$ and adds some element $b$, so $T-t + b = \{b, u_1, \hdots, u_{k-1}\}$. 
    
    \noindent Since $a \neq u_1$, these are at distance 2 unless $b = s$ or $b = a$.

    \smallskip \noindent \textbf{Case 3:} $T$ drops $u_i \neq u_1$ and adds $b$. 
    
    \noindent Then $T-u_i+b$ contains $t$ and $u_1$, neither of which are contained in $S-u_1+a$, so $T-u_i+b$ is at distance at least 2 from $S-u_1+a$.
\end{proof}

We define the following relevant events:
\begin{align}
    \mathcal{F}_{u,a} &:= \{S \text{ drops } u \text{ and adds } a \in A_u\} &&\hspace{-2cm}\text{for any } u \in J, a \in A_u\\
    \mathcal{B}_u &:= \{T \text{ drops } u \text{ and adds } s\} &&\hspace{-2cm}\text{for any }u \in J,\\
    \mathcal{E}_a&:= \{T \text{ drops } t \text{ and adds } a\} &&\hspace{-2cm}\text{for any }a \in N(T-t).
\end{align} The above proposition shows \begin{align}\label{eq:probtwo}
    \mathbb{P}(d(X,Y) \geq 2) \geq \sum_{u \in J} \sum_{a \in A_u} \mathbb{P}(\mathcal{F}_{u,a}) - \mathbb{P}(\mathcal{F}_{u,a} \cap \mathcal{B}_u) - \mathbb{P}(\mathcal{F}_{u,a} \cap \mathcal{E}_s) - \mathbb{P}(\mathcal{F}_{u,a} \cap \mathcal{E}_a).
\end{align}

Now we consider events where $d(X,Y) = 0$. The following table gives all possible pairs of exchanges by $S$ and $T$ so that $d(X,Y) = 0$. 

\begin{center}
\begin{tabular}{ | m{4cm}| m{6.5cm} | } 
 \hline
 Exchange by $S$ & Exchange by $T$ \\ \hline \hline
 $S$ drops $s$, adds $t$ & $T$ drops $u$, adds $u$ for any $u \in T$ \\
 $S$ drops $s$, adds $a \neq t$ & $T$ drops $t$, adds $a$\\
 $S$ drops $u \neq s$, adds $u$ & $T$ drops $t$, adds $s$ \\
 $S$ drops $u \in J$, adds $t$ & $T$ drops $u$, adds $s$ \\
 \hline
\end{tabular}
\end{center} We define the additional relevant events, based on the exchange made by $S$:
\begin{align}
    \mathcal{C}_{1,a} &:= \{S \text{ drops } s, \text{ adds } a\}, \\
    \mathcal{C}_{2,u} &:= \{S \text{ drops } u, \text{ adds } u\}  &&\hspace{-3cm} \text{for any } u \in S, u \neq s,\\
    \mathcal{C}_{3,u} &:= \{S \text{ drops } u, \text{ adds } t\} &&\hspace{-3cm} \text{for any } u \in J.
\end{align} Then we have
\begin{align}\label{eq:probzero}
    \mathbb{P}(d(X,Y) = 0) \leq \mathbb{P}(\mathcal{C}_{1,t}) + \sum_{a \neq t} \mathbb{P}(\mathcal{C}_{1,a} \cap \mathcal{E}_{a}) + \sum_{\substack{u \in S \\ u \neq s}} \mathbb{P}(\mathcal{C}_{2,u} \cap \mathcal{E}_s) + \sum_{u \in J} \mathbb{P}(\mathcal{C}_{3,u} \cap \mathcal{B}_u)  
\end{align}

Now we want to combine these two expressions. First note that if we collect some of the terms in the difference between \ref{eq:probzero} and \ref{eq:probtwo}, we can simplify as follows:
\begin{align}
    \mathbb{P}(\mathcal{C}_{1,t}) + \sum_{a \neq t} \mathbb{P}(\mathcal{C}_{1,a} \cap \mathcal{E}_{a}) + \sum_{\substack{u \in S \\ u \neq s}} \mathbb{P}(\mathcal{C}_{2,u} \cap \mathcal{E}_s) + \sum_{u \in J} \sum_{a \in A_u} (\mathbb{P}(\mathcal{F}_{u,a} \cap \mathcal{E}_s) + \mathbb{P}(\mathcal{F}_{u,a} \cap \mathcal{E}_a)) \\ \leq \mathbb{P}(\mathcal{C}_{1,t}) + \sum_{\substack{a \in N(T-t) \\ a \neq t}} \mathbb{P}(\mathcal{E}_a) = \frac{1}{k} \cdot \left (\frac{1}{\#N(S-s)} + \frac{\#N(S-s) - 1}{\#N(S-s)} \right )= \frac{1}{k}.
\end{align}
Passing to the full expression we have
\begin{align}
    \mathbb{P}(d(X,Y) = 0) - \mathbb{P}(d(X,Y) \geq 2) &\leq \frac{1}{k} + \sum_{u \in J} \mathbb{P}(\mathcal{C}_{3,u} \cap \mathcal{B}_u) + \sum_{u \in J} \sum_{a \in A_u} \mathbb{P}(\mathcal{F}_{u,a} \cap \mathcal{B}_u) - \sum_{u \in J} \sum_{u \in A_u} \mathbb{P}(\mathcal{F}_{u,a}) \\
    &\leq \frac{1}{k} + \sum_{u \in J} \mathbb{P}(\mathcal{B}_u) - \sum_{u \in J}\sum_{u \in A_u} \mathbb{P}(\mathcal{F}_{u,a}).
\end{align}
Explicitly computing these probabilities, we have \begin{align}
    \kappa &\leq \frac{1}{k} + \frac{1}{k}\sum_{u\in J}\left (\frac{1}{\#N(T-u)} - \frac{\#A_u}{\#N(S-u)} \right ).
\end{align}
\end{proof}

An immediate corollary of the above is a sufficient condition for negative curvature:
\begin{corollary}
    The basis exchange walk on $\mathcal{M}$ is negatively curved if there exist  bases $S \sim T$ such that \[\sum_{u\in J}\left (\frac{\#A_u}{\#N(S-u)} - \frac{1}{\#N(T-u)} \right ) > 1.\]
\end{corollary}

\section{Examples of negatively curved basis exchange walks}\label{sec:negex}
By Corollary \ref{cor:rank2}, all basis exchange walks over rank-$2$ matroids are positively curved. However, as soon as we move to rank-$3$ matroids, this is no longer true as shown in Section \ref{sec:rank3}. One might then hope that certain classes of matroids, for example regular matroids, induce non-negatively curved basis exchange walks. In Section \ref{sec:graphiccounter}, we show this is not the case: we give a graphic matroid (thus also a regular matroid) with a negatively curved basis exchange walk.

\subsection{Rank-3 matroid}\label{sec:rank3}
The bound from Theorem \ref{thm:ub} says that in order to ensure negative curvature, it suffices to have $S$ and $T$ so that $N(S-u)$ is very different from $N(T-u)$ for many $u$. We construct a matroid with this property.

Let $E = \{s, t, u, u', v_1, \hdots, v_5, w_1, \hdots, w_5\}$. We define a matroid over $E$ by defining the bases:
\begin{itemize}
    \item $S = \{s,u,u'\}$,
    \item $T = \{t,u,u'\}$,
    \item $\{s,t,u\}, \{s,t,u'\}$,
    \item $\{s, u, v_i\}, \{s, u', v_i\}$ for all $i$,
    \item $\{t, u, w_i\}, \{t, u', w_i\}$ for all $i$,
    \item $\{u,u',v_i\}, \{u,u',w_i\}$ for all $i$, and
    \item $\{u, v_i, w_j\}, \{u', v_i, w_j\}$ for all $i, j$.
\end{itemize} This is a $\mathbb{R}$-linear matroid, with the elements identified with the following vectors in $\mathbb{R}^3$: \[s = \begin{bmatrix} 1 \\ 0 \\ 0 \end{bmatrix}, \hspace{2mm} t = \begin{bmatrix}
    0 \\ 1 \\ 0
\end{bmatrix}, \hspace{2mm} u = \begin{bmatrix}
    0 \\ 0 \\ 1
\end{bmatrix}, \hspace{2mm} u' = \begin{bmatrix}
    1 \\ 1 \\ 1
\end{bmatrix}, \hspace{2mm} v_i = \begin{bmatrix}
    0 \\ 1 \\ 0
\end{bmatrix}, \hspace{2mm} w_i = \begin{bmatrix}
    1 \\ 0 \\ 0
\end{bmatrix}.\]
\smallskip

\noindent We have $J = \{u, u'\}$, $\#A_j = 5$, and $N(S-j) = N(T-j)= 7$ for $j \in J$, so by Theorem \ref{thm:ub} \begin{align}
    \kappa \leq \frac{1}{3} + \frac{1}{3} \cdot \left (\frac{2}{7} - \frac{2\cdot 5}{7} \right ) = -\frac{1}{21}.
\end{align}

\subsection{Graphic matroid}\label{sec:graphiccounter}
We want to construct a graphic matroid with spanning trees $S$ and $T$ so that $N(S-u)$ and $N(T-u)$ are very different for many $u$. We look at the graphic matroid induced by $K_6$; in order to make the set $J$ as large as possible, we choose $S$ and $T$ to be paths. We label the ``outer" edges of $K_6$ as in Figure 3, and consider the spanning trees $S = \{s,1,2,3,4\}$ and $T = \{t, 1,2,3,4\}$.

\begin{figure}[h]
\centering
\begin{tikzpicture}[scale=0.7, thick]

\foreach \i in {1,...,6} {
    \coordinate (P\i) at ({90 + (\i - 1) * 60} : 1);
}

\foreach \i in {1,...,6} {
    \pgfmathtruncatemacro{\j}{mod(\i,6) + 1}
    \draw[black, very thick] (P\i) -- (P\j);
}

\foreach \i in {1,...,6} {
    \foreach \j in {1,...,6} {
        \ifnum\i<\j
            \pgfmathtruncatemacro{\diff}{\j - \i}
            \ifnum\diff>1
                \ifnum\diff<5
                    \draw[gray, thick] (P\i) -- (P\j);
                \fi
            \fi
        \fi
    }
}

\node[red] at ($(P1) + (-0.6,-0.1)$) {1};
\node[red] at ($(P2) + (-0.2,-0.5)$) {2};
\node[red] at ($(P3) + (0.3,-0.4)$) {t};
\node[red] at ($(P4) + (0.55,0.1)$) {3};
\node[red] at ($(P5) + (0.2,0.4)$) {4};
\node[red] at ($(P6) + (-0.3,0.35)$) {s};

\end{tikzpicture}
\caption{Labeling of outer edges of $K_6$.}
\end{figure}

We count the relevant quantities for Theorem \ref{thm:ub}. In the following table, let $i$ be the element dropped in the ``down" step of the down-up walk.

\begin{center}
\begin{tabular}{ | m{0.2cm} | m{1.7cm}| m{1.7cm} | m{1.7cm} | } 
 \hline
 $i$ & $\# N(S-i)$ & $\# N(T-i)$ & $\# A_i $\\ \hline \hline
 1 & $2 \cdot 4$ & $1 \cdot 5$ & 5\\ \hline
 2 & $1 \cdot 5$ & $2 \cdot 4$ & 2\\ \hline
 3 & $1 \cdot 5$ & $2 \cdot 4$ & 2\\ \hline
 4 & $2 \cdot 4$ & $1 \cdot 5$ & 5\\
 \hline
\end{tabular}
\end{center}

\noindent By Theorem \ref{thm:ub} we have \begin{align}
    \kappa &\leq \frac{1}{5} + \frac{1}{5} \left (\frac{2}{5} + \frac{2}{8} - \frac{2 \cdot 5}{8} - \frac{2\cdot 2}{5} \right ) = -\frac{2}{25}.
\end{align}

\section*{Acknowledgements}
The author would like to thank Nikhil Srivastava for introducing her to this question and Zack Stier for helpful conversations and comments. 

\bibliographystyle{alpha}
\bibliography{matroids}

@book{LPW,
  title = {Markov Chains and Mixing Times},
  ISBN = {9781470442323},
  url = {http://dx.doi.org/10.1090/mbk/107},
  DOI = {10.1090/mbk/107},
  publisher = {American Mathematical
                    Society},
  author = {Levin,  David and Peres,  Yuval},
  year = {2017},
  month = oct 
}

@inproceedings{kldivbew,
  title = {Modified log-Sobolev Inequalities for Strongly Log-Concave Distributions},
  url = {http://dx.doi.org/10.1109/FOCS.2019.00083},
  DOI = {10.1109/focs.2019.00083},
  booktitle = {2019 IEEE 60th Annual Symposium on Foundations of Computer Science (FOCS)},
  publisher = {IEEE},
  author = {Cryan,  Mary and Guo,  Heng and Mousa,  Giorgos},
  year = {2019},
  month = nov,
  pages = {1358–1370}
}

@inproceedings{couplingtospecind,
  title = {On Mixing of Markov Chains: Coupling,  Spectral Independence,  and Entropy Factorization},
  ISBN = {9781611977073},
  url = {http://dx.doi.org/10.1137/1.9781611977073.145},
  DOI = {10.1137/1.9781611977073.145},
  booktitle = {Proceedings of the 2022 Annual ACM-SIAM Symposium on Discrete Algorithms (SODA)},
  publisher = {Society for Industrial and Applied Mathematics},
  author = {Blanca,  Antonio and Caputo,  Pietro and Chen,  Zongchen and Parisi,  Daniel and Štefankovič,  Daniel and Vigoda,  Eric},
  year = {2022},
  month = jan,
  pages = {3670–3692}
}

@inproceedings{couplingliu,
  doi = {10.4230/LIPICS.APPROX/RANDOM.2021.32},
  url = {https://drops.dagstuhl.de/entities/document/10.4230/LIPIcs.APPROX/RANDOM.2021.32},
  author = {Liu,  Kuikui},
  language = {en},
  title = {From Coupling to Spectral Independence and Blackbox Comparison with the Down-Up Walk},
  booktitle = {APPROX-RANDOM},
  year = {2021},
  copyright = {Creative Commons Attribution 4.0 International license}
}

@inproceedings{ALOVIV,
  series = {STOC ’21},
  title = {Log-concave polynomials IV: approximate exchange,  tight mixing times,  and near-optimal sampling of forests},
  url = {http://dx.doi.org/10.1145/3406325.3451091},
  DOI = {10.1145/3406325.3451091},
  booktitle = {Proceedings of the 53rd Annual ACM SIGACT Symposium on Theory of Computing},
  publisher = {ACM},
  author = {Anari,  Nima and Liu,  Kuikui and Gharan,  Shayan Oveis and Vinzant,  Cynthia and Vuong,  Thuy-Duong},
  year = {2021},
  month = jun,
  pages = {408–420},
  collection = {STOC ’21}
}

@article{ORcurv,
  title = {Ricci curvature of Markov chains on metric spaces},
  volume = {256},
  ISSN = {0022-1236},
  url = {http://dx.doi.org/10.1016/j.jfa.2008.11.001},
  DOI = {10.1016/j.jfa.2008.11.001},
  number = {3},
  journal = {Journal of Functional Analysis},
  publisher = {Elsevier BV},
  author = {Ollivier,  Yann},
  year = {2009},
  month = feb,
  pages = {810–864}
}

\end{document}